\documentclass[10pt,reqno]{amsart}
\usepackage{bbm}
\usepackage{mathrsfs}
\usepackage{amsfonts} 
\usepackage[dvipsnames,usenames]{color}
\usepackage{graphicx,latexsym,bm,amsmath,amssymb,verbatim,multicol,lscape}
\newtheorem{thm}{Theorem} [section]

\newtheorem{lem}[thm]{Lemma}

\theoremstyle{definition}

\theoremstyle{remark}

\newtheorem{con}[thm]{Conjecture}
\numberwithin{equation}{section}

\begin{document}
\title[Multiple reciprocal sums and multiple reciprocal
star sums are never integers]
{Multiple reciprocal sums and multiple reciprocal star sums
of polynomials are almost never integers}
\begin{abstract}
Let $n$ and $k$ be integers such that $1\le k\le n$
and $f(x)$ be a nonzero polynomial of integer coefficients
such that $f(m)\ne 0$ for any positive integer $m$. For any $k$-tuple
$\vec{s}=(s_1, ..., s_k)$ of positive integers, we define
$$H_{k,f}(\vec{s}, n):=\sum\limits_{1\leq i_{1}<\cdots<i_{k}\le n}
\prod\limits_{j=1}^{k}\frac{1}{f(i_{j})^{s_j}}$$
and
$$H_{k,f}^*(\vec{s}, n):=\sum\limits_{1\leq i_{1}\leq \cdots\leq i_{k}\leq n}
\prod\limits_{j=1}^{k}\frac{1}{f(i_{j})^{s_j}}.$$
If all $s_j$ are 1, then let $H_{k,f}(\vec{s}, n):=H_{k,f}(n)$
and $H_{k,f}^*(\vec{s}, n):=H_{k,f}^*(n)$. Hong and Wang refined
the results of Erd\"{o}s and Niven, and of Chen and Tang
by showing that $H_{k,f}(n)$ is not an integer if $n\geq 4$
and $f(x)=ax+b$ with $a$ and $b$ being positive integers.
Meanwhile, Luo, Hong, Qian and Wang established the similar
result when $f(x)$ is of nonnegative integer coefficients
and of degree no less than two. For any $k$-tuple
$\vec{s}=(s_1, ..., s_k)$ of positive integers, Pilehrood,
Pilehrood and Tauraso proved that $H_{k,f}(\vec{s},n)$ and
$H_{k,f}^*(\vec{s},n)$ are nearly never integers if $f(x)=x$.
In this paper, we show that if $f(x)$ is a nonzero polynomial
of nonnegative integer coefficients such that either $\deg f(x)\ge 2$
or $f(x)$ is linear and $s_j\ge 2$ for all integers $j$ with
$1\le j\le k$, then $H_{k,f}(\vec{s}, n)$ and $H_{k,f}^*(\vec{s}, n)$
are not integers except for the case $f(x)=x^{m}$ with $m\geq1$
being an integer and $n=k=1$, in which case, both of $H_{k,f}(\vec{s}, n)$
and $H_{k,f}^*(\vec{s}, n)$ are integers. Furthermore, we prove
that if $f(x)=2x-1$, then both $H_{k,f}(\vec{s}, n)$
and $H_{k,f}^*(\vec{s}, n)$ are not integers except when $n=1$,
in which case $H_{k,f}(\vec{s}, n)$ and $H_{k,f}^*(\vec{s}, n)$
are integers. The method of the proofs is analytic and $p$-adic.
\end{abstract}
\author[Q.Y. Yin]{Qiuyu Yin}
\address{Mathematical College, Sichuan University, Chengdu 610064, P.R. China}
\email{yinqiuyu26@126.com}
\author[S.F. Hong]{Shaofang Hong$^*$}
\address{Mathematical College, Sichuan University, Chengdu 610064, P.R. China}
\email{sfhong@scu.edu.cn; s-f.hong@tom.com; hongsf02@yahoo.com}
\author[L.P. Yang]{Liping Yang}
\address{Mathematical College, Sichuan University, Chengdu 610064, P.R. China}
\email{yanglp2013@126.com}
\author[M. Qiu]{Min Qiu}
\address{Mathematical College, Sichuan University, Chengdu 610064, P.R. China}
\email{qiumin126@126.com}
\thanks{$^*$S.F. Hong is the corresponding author and was supported
partially by National Science Foundation of China Grant \#11771304
and \#11371260.}
\maketitle

\section{Introduction}

Let $n$ and $k$ be integers with $1\le k\le n$ and $f(x)$ be
a polynomial of integer coefficients such that $f(m)\ne 0$
for any positive integer $m$. For any $k$-tuple
$\vec{s}=(s_1, ..., s_k)$ of positive integers, one defines
the {\it multiple reciprocal sum}, denoted by $H_{k,f}(\vec{s}, n)$,
and the {\it multiple reciprocal star sum}, denoted by
$H_{k,f}^*(\vec{s}, n)$, of $f(x)$ as follows:
$$H_{k,f}(\vec{s}, n):=\sum\limits_{1\leq i_{1}<\cdots<i_{k}\le n}
\prod\limits_{j=1}^{k}\frac{1}{f(i_{j})^{s_j}}$$
and
$$H_{k,f}^*(\vec{s}, n):=\sum\limits_{1\leq i_{1}\leq \cdots\leq i_{k}\leq n}
\prod\limits_{j=1}^{k}\frac{1}{f(i_{j})^{s_j}}.$$
For brevity, if $\vec{s}=(1, ..., 1)$, then we write
$H_{k,f}(\vec{s}, n):=H_{k,f}(n)$ and
$H_{k,f}^*(\vec{s}, n):=H_{k,f}^*(n)$ that are called the
{\it multiple harmonic sum} and {\it multiple harmonic star sum}
of $f(x)$, respectively. Such sums are closely related to the
so-called multiple zeta functions that are nested generalizations
of the Riemann zeta function to the multiple variable setting.
For the multiple zeta functions, the readers are referred
to \cite{[Z]}-\cite{[Z2]}.

For the case $\vec{s}=(1, ..., 1)$,
it is well known that if $n\geq 2$ and $f(x)=x$, then
$H_{1,f}(n)$ is not an integer. If $n\ge 2$ and $f(x)=ax+b$,
where $a$ and $b$ are positive integers, Erd\"{o}s and Niven \cite{[EN]}
showed in 1946 that there is only a finite number of integers $n$
such that $H_{k,f}(n)$ can be an integer. Chen and Tang
\cite{[CT]} proved that $H_{k,f}(n)$ cannot be an integer
except for either $n=k=1$ or $n=3$ and $k=2$ if $f(x)=x$.
This result was generalized by Yang, Li, Feng and Jiang \cite{[YLFJ]}.
Wang and Hong \cite{[WH]} proved that $H_{k,f}(n)$ cannot
be an integer if $f(x)=2x-1$ and $n\geq 2$. Consequently,
Hong and Wang \cite{[HW]} extended the results of \cite{[CT]}
and \cite{[WH]} by showing that $H_{k,f}(n)$ is not
an integer if $n\geq 4$ and $f(x)=ax+b$ with $a$ and $b$
being positive integers. Later on, Luo, Hong, Qian and Wang
\cite{[LHQW]} proved that the similar result holds if $f(x)$
is of nonnegative integer coefficients and of degree no less
than two.

Now we let $\vec{s}=(s_1, ..., s_k)$ be any $k$-tuple
of positive integers. Naturally, the following interesting question arises:
If $f(x)$ is an arbitrary polynomial of nonnegative integer coefficients,
are the similar results true for both of $H_{k,f}(\vec{s}, n)$ and
$H_{k,f}^*(\vec{s}, n)$? Recently, Pilehrood et al \cite{[KH]} showed that
if $f(x)=x$, then $H_{k,f}(\vec{s}, n)$ and $H_{k,f}^*(\vec{s}, n)$
are nearly never integers. However, this problem is kept open if $f(x)\ne x$
is any polynomial of nonnegative integer coefficients, see [8, Problem 1].

In this paper, we concentrate on the integrality of $H_{k,f}(\vec{s}, n)$
and $H_{k,f}^*(\vec{s}, n)$. In fact, by developing the techniques and
ideas in \cite{[HW]}, \cite{[LHQW]} and \cite{[WH]}, we will show that
if either $f(x)$ is of degree at least two, or
$f(x)$ is linear and $s_j\ge 2$ for all integers $j$ with
$1\le j\le k$, then both of $H_{k,f}(\vec{s},n)$ and $H_{k,f}^*(\vec{s},n)$
are not integers with the exception of $f(x)=x^{m}$ with $m\geq1$
being an integer and $n=1$. In other words, the first main result of
this paper can be stated as follows.

\begin{thm}\label{thm}
Let $n$ and $k$ be integers such that $1\le k\le n$.
Let $f(x)$ be a nonzero polynomial of nonnegative integer coefficients
and $\vec{s}=(s_1, ..., s_k)$ be a $k$-tuple of positive integers such
that either $f(x)$ is of degree at least two, or $f(x)$ is linear
and $s_j$ is greater than two for all integers $j$ with
$1\le j\le k$. Then $H_{k,f}(\vec{s}, n)$ and $H_{k,f}^*(\vec{s}, n)$
are not integers except for the case when $f(x)=x^{m}$ with $m\geq1$
being an integer and $n=k=1$, in which case, both of $H_{k,f}(\vec{s}, n)$
and $H_{k,f}^*(\vec{s}, n)$ are integers.
\end{thm}

If $f(x)=ax+b$ with $a$ and $b$ being positive integers,
then \cite{[HW]} tells us that when $s_j=1$ for all integers $j$
with $1\le j\le k$, $H_{k,f}(\vec{s}, n)$ is not an integer. But the
integrality of $H_{k,f}(\vec{s}, n)$ and $H_{k,f}^*(\vec{s}, n)$
is still unknown if $f(x)\ne x$ and at least one of $s_j$ is
strictly greater than 1 and at least one of $s_j$ equals 1.
In this regard, we have the following result that is the
second main result of this paper and
extends the main result of \cite{[WH]}.

\begin{thm}\label{thm2}
Let $k$ and $n$ be positive integers such that $1\le k \le n$.
Let $\vec{s}=(s_1, ..., s_k)$ be a $k$-tuple of positive integers
and let $f(x)=2x-1$. Then both of $H_{k,f}(\vec{s},n)$
and $H_{k,f}^*(\vec{s},n)$ are not integers except when $n=1$,
in which case both of $H_{k,f}(\vec{s},n)$ and $H_{k,f}^*(\vec{s},n)$
are integers.
\end{thm}

The method of the proofs of Theorems 1.1 and 1.2 is analytic
and $p$-adic in character. The paper is organized as follows.
First, in Section 2, we show some preliminary lemmas which
are needed for the proofs of Theorems \ref{thm} and \ref{thm2}.
Then in Section 3 and Section 4, we present the proofs of Theorem
\ref{thm} and Theorem \ref{thm2}, respectively. Finally,
we consider the integrality problem for any integer coefficients
polynomial $f(x)$ and, in fact, we propose a conjecture as a
conclusion of this paper.

\section{Auxiliary lemmas}

In this section, we present several auxiliary lemmas that are needed
in the proofs of Theorems {\ref{thm}} and {\ref{thm2}}.
We begin with the following results.

\begin{lem}\label{lem1} Let $n$ be a positive integer with $n\ge 2$
and let $f(x)$ be a polynomial of integer coefficients such that
$f(m)>0$ for all positive integers $m$. Then each of the following is true:

{\rm (i).} For any integer $k$ with $1\leq k\leq n-1$, we have
$H_{k+1,f}^*(n)< H_{1,f}^*(n)H_{k,f}^*(n).$

{\rm (ii).} If $H_{1,f}^*(n)<1$, then $0<H_{k,f}^*(n)<1$ for all
integers $k$ with $1\le k\le n$.
\end{lem}

\begin{proof} (i). Let $k$ be an integer such that $1\leq k\leq n-1$. Since
$$H_{k+1,f}^*(n)=\sum_{1\leq i_{1}\leq \cdots\leq i_k\leq i_{k+1}\leq n}
\prod\limits_{j=1}^{k+1}\frac{1}{f(i_{j})}$$
and the coefficients of $f$ are integers such that $f(m)>0$
for all positive integers $m$, it then follows that
\begin{align} \label{eq0}
H_{k+1,f}^*(n)&=\sum\limits_{t=1}^{n}
\frac{1}{f(t)}\sum\limits_{1\leq i_1\leq\cdots \le i_k\leq t}
\prod\limits_{j=1}^{k}\frac{1}{f(i_{j})}\nonumber \\
&< \sum\limits_{t=1}^{n}\frac{1}{f(t)}\sum\limits_{1\leq i_1\leq\cdots\le i_k\leq n}
\prod\limits_{j=1}^{k}\frac{1}{f(i_{j})}({\rm since} \ n\ge2)\nonumber \\
&=\sum\limits_{t=1}^{n}\frac{1}{f(t)}\cdot  H_{k,f}^*(n)\nonumber \\
&=H_{1,f}^*(n)H_{k,f}^*(n)
\end{align}
as required. So part (i) is proved.

(ii). Clearly, $H_{k,f}^*(n)>0$ for all integers $k$ with $1\le k\le n$.
Since $H_{1,f}^*(n)<1$, one has $H_{k-1,f}^*(n)H_{1,f}^*(n)<H_{k-1,f}^*(n)$.
But (\ref{eq0}) gives that $H_{k,f}^*(n)< H_{1,f}^*(n)H_{k-1,f}^*(n)$. So
$H_{k,f}^*(n)<H_{k-1,f}^*(n)$. Namely, $H_{k,f}^*(n)$ is decreasing as $k$
increases. Thus $H_{k,f}^*(n)\le H_{1,f}^*(n)<1$ for all
integers $k$ with $1\le k\le n$. Part (ii) is proved.

This completes the proof of Lemma \ref{lem1}.
\end{proof}

\begin{lem}\label{lem2}Let $n$ and $t$ be positive integers with $n\ge 3$
and let $f(x)$ be a polynomial of integer coefficients such that
$f(2)^2>f(1)f(3)$ and $f(m)>0$ for all positive integers $m$.
Then for any positive integer $k$ with $k\le n-1$, we have
$$H_{k+1,f}^*(n)<\Bigg(\frac{1}{f(1)}
+\sum\limits_{t=3}^{n}\frac{1}{f(t)}\Bigg)H_{k,f}^*(n).$$
\end{lem}

\begin{proof}
Since $f(x)$ is of integer coefficients and $f(m)>0$
for all positive integers $m$, we can deduce that
\begin{align} \label{eq1}
H_{k+1,f}^*(n)&=\sum\limits_{1\leq i_{1}\leq\cdots \le i_k\le i_{k+1}
\leq n}\prod\limits_{j=1}^{k+1}\frac{1}{f(i_{j})}\nonumber\\
&=\sum\limits_{t=1}^n\frac{1}{f(t)}\sum\limits_{1\leq i_{1}\leq \cdots \leq i_{k} \leq t}\prod\limits_{j=1}^{k}\frac{1}{f(i_{j})}\nonumber\\
&=\sum\limits_{t=3}^n\frac{1}{f(t)}\sum\limits_{1\leq i_{1}\leq \cdots \leq i_{k}\leq t}\prod\limits_{j=1}^{k}\frac{1}{f(i_{j})}
+\frac{1}{f(2)}\sum\limits_{1 \leq i_{1}\leq \cdots \leq i_{k}\le 2}\prod\limits_{j=1}^{k}\frac{1}{f(i_{j})}\nonumber\\
&+\frac{1}{f(1)}\sum\limits_{1\leq i_{1}\leq \cdots \leq i_{k}\le 1}\prod\limits_{j=1}^{k}\frac{1}{f(i_{j})}\nonumber\\
&=\sum\limits_{t=3}^n\frac{1}{f(t)}H_{k,t}^*(f)+\sum_{i=0}^k\frac{1}{f(1)^if(2)^{k+1-i}}+\frac{1}{f(1)^{k+1}}\nonumber\\
&\le \sum\limits_{t=3}^n\frac{1}{f(t)}H_{k, n}^*(f)
+\sum_{i=0}^k\frac{1}{f(1)^if(2)^{k+1-i}}+\frac{1}{f(1)^{k+1}} \ ({\rm since} \ t\le n)\nonumber\\
&=\sum\limits_{t=3}^n\frac{1}{f(t)}H_{k, n}^*(f)+\sum_{i=1}^k
\frac{1}{f(1)^if(2)^{k+1-i}}+\frac{1}{f(1)^{k+1}}+\frac{1}{f(2)^{k+1}}.
\end{align}

Notice that the hypothesis that $f(2)^2>f(1)f(3)>0$ implies that
$$\frac{1}{f(1)f(3)}>\frac{1}{f(2)^2}.$$
It follows that
\begin{align} \label{eq2}
\frac{1}{f(2)^{k+1}}<\frac{1}{f(1)f(3)f(2)^{k-1}}.
\end{align}
Then from (\ref{eq1}) and (\ref{eq2}) we derive that
\begin{align*}
H_{k+1,f}^*(n)&<\sum\limits_{t=3}^n\frac{1}{f(t)}H_{k, n}^*(f)+\sum_{i=1}^k\frac{1}{f(1)^if(2)^{k+1-i}}+\frac{1}{f(1)^{k+1}}+\frac{1}{f(1)f(3)f(2)^{k-1}}\\
&< \sum\limits_{t=3}^n\frac{1}{f(t)}H_{k, n}^*(f)+\frac{1}{f(1)} \sum\limits_{1\leq i_{1}\leq\cdots \le i_k\leq n}\prod\limits_{j=1}^{k}\frac{1}{f(i_{j})}\\
&=\Big(\frac{1}{f(1)}+\sum\limits_{t=3}^n\frac{1}{f(t)}\Big)H_{k, n}^*(f)
\end{align*}
as desired. This finishes the proof of Lemma \ref{lem2}.
\end{proof}

\begin{lem}\label{lem3} Let $a$, $b$, $m$ and $n$ be positive integers such that $2\le a\le b$ and $m<n$. Then
\begin{align*}
\sum_{m \le i<j \le n}\frac{1}{i^{a}j^b}<&\frac{1}{2}\big(\zeta(a)\zeta(b)-\zeta(a+b)-\zeta(a)H_{m-1}(b)-\zeta(b)H_{m-1}(a)\\
&+H_{m-1}(a)H_{m-1}(b)+H_{m-1}(a+b)\big)
\end{align*}
with $H_s(t)$ being defined by
$$H_{s}(t):=\sum\limits_{i=1}^{s}\dfrac{1}{i^t}$$
for all positive integers $s$ and $t$.
\end{lem}

\begin{proof} By the definitions of $H_s(t)$ and Riemann zeta function, we can easily derive that
\begin{align*}
\sum_{m \le i<j \le n}\frac{1}{i^{a}j^b}\le &\frac{1}{2}\sum_{m \le i\neq  j \le n}\frac{1}{i^{a}j^b}\\
<&\frac{1}{2}\sum_{i, j\ge m \atop i\neq j}\frac{1}{i^{a}j^b}\\
=&\frac{1}{2}\Big(\sum_{i, j\ge m}\frac{1}{i^{a}j^b}-\sum_{i=m}^\infty \frac{1}{i^{a+b}}\Big)\\
=&\frac{1}{2}\Big(\Big(\sum_{i=m}^\infty\frac{1}{i^{a}}\Big)\Big(\sum_{j=m}^\infty\frac{1}{j^{b}}\Big)-\sum_{i=m}^\infty \frac{1}{i^{a+b}}\Big)\\
=&\frac{1}{2}\big(\big(\zeta(a)-H_{m-1}(a)\big)\big(\zeta(b)-H_{m-1}(b)\big)-\big(\zeta(a+b)-H_{m-1}(a+b)\big)\big)\\
=&\frac{1}{2}\big(\zeta(a)\zeta(b)-\zeta(a+b)-\zeta(a)H_{m-1}(b)-\zeta(b)H_{m-1}(a)\\
&+H_{m-1}(a)H_{m-1}(b)+H_{m-1}(a+b)\big)
\end{align*}
as required. The proof of Lemma \ref{lem3} is complete.
\end{proof}

{\bf Remark.} We notice that the special case when $m=1$ and $a=b=2$
of Lemma \ref{lem3} was used in (2.2) of \cite{[LHQW]}. We also point
out that the special case when $m=3$ and $a=b=2$ of Lemma \ref{lem3}
will be used in the proof of Theorem \ref{thm} below.

\begin{lem}\label{lem4}
Let $a, b, k$ and $n$ be positive integers such that $2\le k \le n$.
Let $\vec{s}=(s_1, ..., s_k)$ be a $k$-tuple of positive integers
and $f(x)=ax+(b-a)$. If $n\le \frac{b}{a}\big(e^{a(\sqrt{2b^2+1}-1)/b}-1\big)$
or $k \ge \frac{e}{a}\log\frac{an+b}{b}+\frac{e}{b}$,
then $0< H_{k,f}(\vec{s},n)<1$.
\end{lem}
\begin{proof}
For any integer $i_j$ such that $1\le i_j\le n$, we have $1\le b\le ai_j+b-a$.
So it is clear that $ H_{k,f}(\vec{s},n)>0$.
To prove Lemma \ref{lem4}, it is sufficient to prove $H_{k,f}(\vec{s},n)<1$.
The hypothesis $s_j \ge 1$ together with $ai_j+b-a\ge 1$
for all integers $j$ with $1\le j\le k$ tells us
$$\frac{1}{(ai_j+b-a)^{s_j}}\le \frac{1}{ai_j+b-a}$$
for all integers $i_j$ with $1 \le i_j \le n$. It then follows that
$$H_{k,f}(\vec{s},n)\le H_{k,f}(n).$$

On the other hand, we have
$$H_{k,f}(n)=\sum\limits_{1\le i_1<\cdots<i_k\le n}\prod\limits_{j=1}^k\frac{1}{ai_j+b-a}=
\sum\limits_{0\le i'_1<\cdots<i'_k\le n-1}\prod\limits_{j=1}^k\frac{1}{ai'_j+b}.$$
Lemma 2.2 in \cite{[HW]} tells us $H_{k,f}(n)<1$ if $n\le \frac{b}{a}\big(e^{a(\sqrt{2b^2+1}-1)/b}-1\big)$
or $k \ge \frac{e}{a}\log\frac{an+b}{b}+\frac{e}{b}$.
Thus one has $H_{k,f}(\vec{s},n)<H_{k,f}(n)<1$ as desired.
So Lemma \ref{lem4} is proved.
\end{proof}

\begin{lem}\label{lem5} \rm{(}\cite{[PD1]},\cite{[PD2]}\rm{)}
\it {For any real number $x\ge 3275$, there is a prime number $p$
satisfying $x< p\le x\big(1+\frac{1}{2\log^2x}\big).$}
\end{lem}

\begin{lem}\label{lem6} Let $n$ be an integer such that $n>62801$.
Then for any integer $k$ with $1\le k\le \frac{e}{2}\log(2n+1)+e$,
there exists a prime $p$ satisfying
$\frac{n}{k+\frac{1}{2}}<p\le \frac{n}{k}$ and $p> 2k$.
\end{lem}

\begin{proof}
We claim that if $n>62801$ and $1\le k\le \frac{e}{2}\log(2n+1)+e$,
then the following two inequalities hold:
\begin{align} \label{eq18}
\frac{n}{k+\frac{1}{2}}\ge 3275,
\end{align}
and
\begin{align} \label{eq19}
\log^2\frac{n}{k+\frac{1}{2}}\ge k.
\end{align}
In fact, if the claim is true, then Lemma \ref{lem5}
tells us that there exists a prime $p$ such that
$$\frac{n}{k+\frac{1}{2}}<p\le \frac{n}{k+\frac{1}{2}}\Big(1+\frac{1}{2\log^2\frac{n}{k+\frac{1}{2}}}\Big)
\le\frac{n}{k+\frac{1}{2}}\Big(1+\frac{1}{2k}\Big)= \frac{n}{k}.$$
Further, it is obvious that $\log(2n+1)<\log2n+1$ if $n\ge 1$. Thus
\begin{align}\label{eq200}
\frac{e}{2}\log2n+\frac{3e}{2}>\frac{e}{2}\log(2n+1)+e \ge k.
\end{align}
 Let
$$h(x)=x-\Big(\frac{e}{2}\log2x+\frac{3e}{2}+\frac{1}{2}\Big)(e\log2x+3e).$$
Then one has
$$h'(x)=1-\Big(\frac{e^2\log2x}{x}+\frac{3e^2+\frac{e}{2}}{x}\Big).$$
But $h(62801)>0$ and $h'(n)>0$ since $n>62801$, so $h(n)>0$.
From (\ref{eq200}), it follows that $n> 2k(k+\frac{1}{2})$,
which implies that $p\ge\frac{n}{k+\frac{1}{2}}>2k$.

Thus, in the following, we only need to prove the claim.
First, we show (\ref{eq18}). Let
$$f(x)=x-3275\Big(\frac{e}{2}\log(2x+1)+e+\frac{1}{2}\Big).$$
Then $f'(x)=1-\frac{3275e}{2x+1}.$
It is easy to check that $f'(x)>0$ if $x>62801$.
Further, we have $f(62801)>0$. Hence, for any integer
$n>62801$, we have $f(n)>f(62801)>0$, which implies that
$$n>3275\Big(\frac{e}{2}\log(2n+1)+e+\frac{1}{2}\Big)>3275\big(k+\frac{1}{2}\big)$$
since $\frac{e}{2}\log(2n+1)+e \ge k$. So (\ref{eq18}) follows immediately.

Now, we show (\ref{eq19}). By (\ref{eq200}), it is enough to show that
\begin{align}\label{eq20}
\Big(\log n-\log\Big(k+\frac{1}{2}\Big)\Big)^2> \frac{e}{2}\log 2n+\frac{3e}{2}.
\end{align}
Moreover, it is not hard to see that (\ref{eq20}) follows from the following inequality
\begin{align}\label{eq21}
\log n-2\log\Big(k+\frac{1}{2}\Big)> \frac{e}{2}+\frac{e}{\log n}\Big(\frac{\log2}{2}+\frac{3}{2}\Big),
\end{align}
Hence our goal is to prove (\ref{eq21}).

Let $$g(x)=x-2\log\Big(\frac{e\log2}{2}+\frac{ex}{2}+\frac{3e}{2}+\frac{1}{2}\Big).$$
Then one has $$g'(x)=1-\frac{2}{x+\log2+3+\frac{1}{e}}.$$
Clearly, $g'(x)>0$ if $x> -1-\log2-\frac{1}{e}$.
But $g(11)>2$. Then one has $g(x)>2$ if $x> 11$.
Since $\log n>11$ if $n>62801$, it follows that
$$g(\log n)=\log n-2\log\Big(\frac{e}{2}\log2n+\frac{3e}{2}+\frac{1}{2}\Big)>2.$$
On the other hand, one can easily check that if $n>62801$, then
$\frac{e}{2}+\frac{e}{\log n}\Big(\frac{\log2}{2}+\frac{3}{2}\Big)<2.$
Thus we get that
\begin{equation}\label{eq200'}
\log n-2\log\Big(\frac{e}{2}\log2n
+\frac{3e}{2}+\frac{1}{2}\Big)>\frac{e}{2}+\frac{e}{\log n}\Big(\frac{\log2}{2}+\frac{3}{2}\Big).
\end{equation}
Then (\ref{eq200}) and (\ref{eq200'}) imply the truth of (\ref{eq21}).
So the claim is proved. This completes the proof of Lemma 2.6.
\end{proof}

For any prime number $p$ and any integer $x$, we let $v_p(x)$ stand for
the largest nonnegative integer $r$ such that $p^r$ divides $x$. In
what follows, we consider the $p$-adic valuation of $H_{k,f}(\vec{s},n)$
under certain conditions if $f(x)=2x-1$.

\begin{lem}\label{lem7}
Let $k$ and $n$ be positive integers such that $1\le k \le n$.
Let $f(x)=2x-1$ and $\vec{s}=(s_1, ..., s_k)$ be a $k$-tuple of
positive integers. If there is a prime number $p$ such that
$\frac{n}{k+\frac{1}{2}}<p\le \frac{n}{k}$ and $p> 2k$, then
$v_p\big(H_{k,f}(\vec{s},n)\big)=-\sum_{i=1}^ks_i.$
\end{lem}
\begin{proof}
Since $p>2k$ and $k\ge 1$, it then follows that $p$
is an odd positive integer. So there
exists a positive integer $r$ such that $2r-1=p$.

The hypothesis $\frac{n}{k+\frac{1}{2}}<p$
tells us that $2n-1<p+2pk-1.$ So we conclude that $\{p,p+2p,\cdots,p+2p(k-1)\}$
are all integers in $\{2i-1\}_{i=1}^{n}$ divisible by $p$.

Using such prime $p$, we can split $H_{k,f}(\vec{s},n)$ as follows: $H_{k,f}(\vec{s},n)=S_1+S_2$, where
$$S_1=\sum\limits_{1\le i_1<\cdots <i_k\le n
\atop p|2i_j-1,\forall1\le j\le k}\prod\limits_{j=1}^k\frac{1}{(2i_j-1)^{s_j}}$$
and
$$S_2=\sum\limits_{1\le i_1<\cdots <i_k\le n
\atop \exists j\ s.t.\ p\nmid 2i_j-1}\prod\limits_{j=1}^k\frac{1}{(2i_j-1)^{s_j}}.$$

We show that $v_p(S_1)=-\sum_{j=1}^k s_j$ in what follows.
Since $\{p,p+2p,\cdots,p+2p(k-1)\}$
are all integers in $\{2i-1\}_{i=1}^{n}$ divisible by $p$,
$S_1$ can be written as follows:
$$S_1=\prod\limits_{j=1}^k\frac{1}{(2pj-p)^{s_j}}=p^{-\sum\limits_{j=1}^ks_j}\prod\limits_{j=1}^k\frac{1}{(2j-1)^{s_j}}.$$
Note that $p>2k$. Then
$$v_p\Big(\prod\limits_{j=1}^k\frac{1}{(2j-1)^{s_j}}\Big)=0.$$
Hence we derive that
$$v_p(S_1)=v_p\Big(p^{-\sum\limits_{j=1}^ks_j}\Big)
+v_p\Big(\prod\limits_{j=1}^k\frac{1}{(2j-1)^{s_j}}\Big)=-\sum\limits_{j=1}^k s_j.$$

Now we count $v_p(S_2)$. We have
\begin{align*}
v_p(S_2)&=v_p\Big(\sum\limits_{1\le i_1<\cdots <i_k\le n
\atop \exists j\ s.t.\ p\nmid 2i_j-1}\prod\limits_{j=1}^k\frac{1}{(2i_j-1)^{s_j}}\Big)\\
&\ge \min\limits_{1\le i_1<\cdots <i_k\le n
\atop \exists j\ s.t.\ p\nmid 2i_j-1}v_p\Big(\prod\limits_{j=1}^k\frac{1}{(2i_j-1)^{s_j}}\Big)\\
&\ge 1-\sum\limits_{j=1}^k s_j.
\end{align*}
Hence, one has $v_p(S_1) < v_p(S_2)$. It then follows that
$$v_p(H_{k,f}(s,n))=v_p(S_1+S_2)=\min\{v_p(S_1), v_p(S_2)\}=-\sum\limits_{j=1}^k s_j$$
as desired. Thus Lemma \ref{lem7} is proved.
\end{proof}

\section{Proof of Theorem \ref{thm}}
It is well known that the values of Riemann zeta function at 2 and 4
are given as follows (see, for instance, \cite{[K]}):
$$\zeta(2)=\sum\limits_{j=1}^{\infty}\frac{1}{j^{2}}=\frac{\pi^{2}}{6} \ {\rm and} \
\zeta(4)=\sum\limits_{j=1}^{\infty}\frac{1}{j^{4}}=\frac{\pi^{4}}{90}.$$
Then $1<\zeta(2)<2$ and
\begin{align} \label{eq11}
\frac{1}{2}<\sum\limits_{j=1}^\infty\frac{1}{j^2+2}<\zeta(2)-\frac{3}{4}<1.
\end{align}
We can now prove Theorem \ref{thm} as follows.\\

{\it Proof of Theorem \ref{thm}.}
First of all, we treat $H_{k,f}(\vec{s},n)$.
Since $f(x)$ is a polynomial of nonnegative integer coefficients
and either $\deg f(x)\ge2$ or $f(x)$ is linear and $s_j \ge2$ for
all integers $j$ with $1\le j\le k$, it follows that for any
positive integer $r$, we have $f(r)^{s_j}\ge r^2$. Then one deduces
that if $k\ge2$, then
\begin{align}\label{eq9'}
H_{k,f}(\vec{s},n)=\sum\limits_{1\leq i_{1}<\cdots< i_{k}\leq n}
\prod\limits_{j=1}^{k}\frac{1}{f(i_{j})^{s_j}}
\le \sum\limits_{1\leq i_{1}<\cdots< i_{k}\leq n}
\prod\limits_{j=1}^{k}\frac{1}{i_j^2}=H_{k,h}(n),
\end{align}
with $h(x):=x^2$. But it was proved in \cite{[LHQW]} that $0<H_{k, q}(n)<1$ for
any polynomial $q(x)$ of nonnegative integer coefficients and of degree at least two.
Hence $0<H_{k,h}(n)<1$. It follows from (\ref{eq9'}) that $0<H_{k,f}(\vec{s},n)<1$.
In other words, $H_{k,f}(\vec{s},n)$ is not an integer if $k\ge 2$.
On the other hand, it is clear that $H_{1,f}(\vec{s},n)=H_{1,f}^*(\vec{s},n)$.
So it remains to deal with the integrality of $H_{k,f}^*(\vec{s},n)$ that
will be done in what follows.

Let $n=1$. If $f(x)=x^m$ with $m\ge 1$ being an integer,
then $H_{1,f}^*(\vec{s},1)=1$ is an integer. Otherwise, one has $H_{1,f}^*(\vec{s},1)=\frac{1}{f(1)^{s_1}}<1$
that infers $H_{1,f}^*(\vec{s},1)$ is not an integer. In the remaining of the proof, we
always let $n>1$. We divide the proof into the following four cases.

{\sc Case 1.} $f(x)$ is a monomial. One may write $f(x)=a_mx^m$ with $m\geq 1$ and $a_m\ge 1$.
Then by Bertrand's postulate, there is
at least one prime $p$ such that $\frac{n}{2}<p\leq n$, i.e. $p\leq n<2p$.
So $p$ cannot divide any integer between 1 and $n$ different from $p$. Hence
\begin{align*}
a_m^{\sum\limits_{i=1}^{k}s_i} \cdot H_{k,f}^*(\vec{s},n)
&=\underset{1\leq i_1\leq \cdots \leq i_k\leq n,\exists i_j \neq p}{\sum}\frac{1}{(i_{1})^{s_1m}
\cdots (i_{k})^{s_km}}+\sum\limits_{i_1=\cdots=i_k=p}\frac{1}{(i_{1})^{s_1m}\cdots (i_{k})^{s_km}}\\
&:=\frac{a}{bp^e}+\frac{1}{p^{m\sum\limits_{i=1}^{k}s_i}},
\end{align*}
with $a, b$ and $e$ being positive integers such that $(a,b)=(a,p)=(b,p)=1$
and $0\le e<m\sum\limits_{i=1}^{k}s_i$. Then one can deduce that
\begin{align} \label{eq9}
a_m^{\sum\limits_{i=1}^{k}s_i} \cdot bp^{m\sum\limits_{i=1}^{k}s_i}
\cdot H_{k,f}^*(\vec{s},n)=ap^{m(\sum\limits_{i=1}^{k}s_i)-e}+b.
\end{align}

Suppose that $H_{k,f}^*(\vec{s},n)$ is an integer.
Then $p$ divides each of
$$bp^{m\sum\limits_{i=1}^{k}s_i} \cdot H_{k,f}^*(\vec{s},n)$$
and
$$ap^{m\big(\sum\limits_{i=1}^{k}s_i\big)-e}.$$
Then from (\ref{eq9}) one can read that $p$ divides $b$,
which contradicts to the fact that $p \nmid b$. So $H_{k,f}^*(\vec{s},n)$
must be non-integer. Theorem 1.1 is proved in this case.

{\sc Case 2.} $f(x)$ holds degree at least two and contains at
least two terms and $f(x)\ne x^2+1$. Claim that $H_{1,f}^*(n)<1$.
Then by Lemma 2.1, one knows that $0<H_{k,f}^*(n)<1$ for all
integers $k$ with $1\le k\le n$. Since the coefficients of $f(x)$
are nonnegative integers, for any positive integer $r$ and $t$,
we have $f(r)\ge 1$, so $f(r)^t \ge f(r)$. Hence we can deduce that
$$0<H_{k,f}^*(\vec{s},n)\le H_{k,f}^*(n)<1.$$
In other words, $H_{k,f}^*(\vec{s},n)$ is not an integer for all
integers $k$ with $1\le k\le n$. It remains to show
the truth of the claim that will be done in what follows.

One lets $f(x)=a_mx^m+a_{m-1}x^{m-1}+\cdots+a_1x+a_0$
with $m\ge 2$ being an integer, $a_m\ge 1$,
$\max(a_0, \cdots, a_{m-1})\geq 1$ and $f(x)\ne x^2+1$.
Consider the following subcases.

{\sc Case 2.1.} $m=2$, $a_1=0$, $ a_0 \ge 1$ and $\max(a_2, a_0)\geq 2$.
Then for any positive integer $j$, one can derive that
$f(j)\ge a_2 j^2+a_0 \geq j^2+2$. It follows from (\ref{eq11}) that
\begin{align*}
H_{1,f}^*(n)=\sum\limits_{j=1}^{n}\frac{1}{f(j)}< \sum\limits_{j=1}^\infty\frac{1}{f(j)}
\le \sum\limits_{j=1}^\infty\frac{1}{j^2+2}<1.
\end{align*}
Hence the claim is true in this case.

{\sc Case 2.2.} $m=2$ and $ a_1\ge 1$. Since $a_2\ge 1$ and $a_1\ge 1$, we deduce that
$f(j)\ge j^2+j$ for any positive integer $j$. So
\begin{align} \label{eq4}
H_{1,f}^*(n)=\sum\limits_{j=1}^{n}\frac{1}{f(j)}\le \sum\limits_{j=1}^{n}\frac{1}{j(j+1)}
=\sum\limits_{j=1}^{n}\Big(\frac{1}{j}-\frac{1}{j+1}\Big)=1-\frac{1}{n+1}<1
\end{align}
as claimed. The claim holds in this case.

{\sc Case 2.3.} $m\geq 3$. Since $\max(a_0, \cdots, a_{m-1})\geq 1$,
we derive that $f(j)\geq j^3+1\ge j^2+j$ for any positive integer $j$.
Thus (\ref{eq4}) keeps valid. So the claim is proved in this case.

{\sc Case 3.} $f(x)=x^2+1$. Notice that $s_1$ is the first component
of $\vec{s}$. We divide the proof into the following subcases.

{\sc Case 3.1.} $k=1$ and $s_1=1$. Then we have
$$H_{1,f}^*(\vec{s},n)=\sum\limits_{i=1}^n\dfrac{1}{i^2+1}$$
increases as $n$ increases. By some computations,
we find that $H_{1,f}^*(\vec{s},12)<1$
and $H_{1,f}^*(\vec{s},13)>1$. So if $n\ge 13$, then
$$1<H_{1,f}^*(\vec{s},13)\leq H_{1,f}^*(\vec{s},n)<\sum_{i=1}^n\frac{1}{i^2}<\zeta(2)<2.$$
On the other hand, if $2\le n\le 12$, then
$$\frac{1}{2}< H_{1,f}^*(\vec{s},n)\le H_{1,f}^*(\vec{s},12)<1.$$
So we can conclude that $H_{1,f}^*(\vec{s}, n)$ is not an integer in this case.

{\sc Case 3.2.} $k=1$ and $s_1 >1$. Clearly, one has
$$0<H_{1,f}^*(\vec{s},n)=\sum\limits_{i=1}^n\frac{1}{f(i)^{s_1}}
\le \sum\limits_{i=1}^n\frac{1}{(i^2+1)^2}
<\frac{1}{4}+\sum\limits_{i=2}^n\frac{1}{i^4}<\zeta(4)-\frac{3}{4}<1.$$
In other words, $H_{k,f}^*(\vec{s},n)$ is not an integer if $k=1$ and $s_1 >1$.

{\sc Case 3.3.} $k >1$. Since $f(r)^t\ge f(r)>0$ for any positive integers $r$
and $t$, we can deduce that $0<H_{k,f}^*(\vec{s},n)\le H_{k,f}^*(n)$. We claim that
$H_{k,f}^*(n)<1$ if $k>1$. Then it follows that $0<H_{k,f}^*(\vec{s},n)<1$ which
means that $H_{k,f}^*(\vec{s},n)$ is not an integer if $k>1$. In the following,
we show the truth of the claim. Its proof is divided into two subcases.

{\sc Case 3.3.1.} $k=2$. Then $n\ge 2$. If $n=2$, then we can easily
compute that $H_{2,f}^*(2)=\frac{39}{100}<1$ as claimed. Now let $n\ge 3$.
Then
\begin{align*}
0<H_{2,f}^*(n)=&\sum\limits_{1\leq i\leq j\leq n}\frac{1}{(i^2+1)(j^2+1)}\\
=&\sum\limits_{i=1}^n\frac{1}{(i^2+1)^2}+\sum\limits_{1\leq i<j\leq n}\frac{1}{(i^2+1)(j^2+1)}\\
=&\frac{1}{4}+\frac{1}{25}+\sum\limits_{i=3}^n\frac{1}{i^4+2i^2+1}+\sum_{j=2}^n\frac{1}{2(j^2+1)}\\
&+\sum_{j=3}^n\frac{1}{5(j^2+1)}+\sum\limits_{3\leq i<j\leq n}\frac{1}{(i^2+1)(j^2+1)}.
\end{align*}
Since
\begin{align}\label{eq80}\sum\limits_{i=3}^n\frac{1}{i^4+2i^2+1}<\sum_{i=3}^\infty\frac{1}{i^4}=
\sum_{i=1}^\infty\frac{1}{i^4}-1-\frac{1}{16}=\zeta(4)-\frac{17}{16},
\end{align}

\begin{align}\label{eq5}\sum_{j=2}^n\frac{1}{2(j^2+1)}
<\sum_{j=2}^\infty\frac{1}{2j^2}=\frac{1}{2}\big(\zeta(2)-1\big),
\end{align}
\begin{align}\label{eq6}
\sum_{j=3}^n\frac{1}{5(j^2+1)}<\sum_{j=3}^\infty\frac{1}{5j^2}=\frac{1}{5}\Big(\zeta(2)-\frac{5}{4}\Big)
\end{align}
and Lemma 2.3 tells us that
\begin{align}\label{eq7}
\sum\limits_{3\leq i<j\leq n}\frac{1}{(i^2+1)(j^2+1)}
<&\sum\limits_{3\leq i<j\leq n}\frac{1}{i^2j^2}\nonumber\\
<&\frac{1}{2}\Big(\zeta(2)^2-\frac{5}{2}\zeta(2)-\zeta(4)+\frac{25}{16}+\frac{17}{16}\Big)\nonumber\\
=&\frac{1}{2}\Big(\zeta(2)^2-\frac{5}{2}\zeta(2)-\zeta(4)+\frac{21}{8}\Big),
\end{align}
by (\ref{eq80}),(\ref{eq5}),(\ref{eq6}) and (\ref{eq7}), one can derive that
\begin{align*}
H_{2,f}^*(n)&<
\frac{1}{4}+\frac{1}{25}+\Big(\zeta(4)-\frac{17}{16}\Big)
+\frac{1}{2}\big(\zeta(2)-1\big)+\frac{1}{5}\Big(\zeta(2)-\frac{5}{4}\Big)\\
&+\frac{1}{2}\Big(\zeta(2)^2-\frac{5}{2}\zeta(2)-\zeta(4)+\frac{21}{8}\Big)\\
&=\frac{1}{2}\zeta(4)+\frac{1}{2}\zeta(2)^2-\frac{11}{20}\zeta(2)-\frac{21}{100}\\
&=\frac{7\pi^4}{360}-\frac{11\pi^2}{120}-\frac{21}{100}<1.
\end{align*}

Thus one can conclude that $H_{2,f}^*(n)<1$ is not an integer if $n\ge 2$.
The claim is proved in this case.

{\sc Case 3.3.2.} $k\geq 3$. Since $f(x)=x^2+1$, it follows that if $n=3$, then
$$
0<\frac{1}{f(1)}+\sum\limits_{i=3}^n\frac{1}{f(i)}=\frac{1}{1^2+1}+\frac{1}{3^2+1}=\frac{3}{5}<1,
$$
and if $n\ge 4$, then
\begin{align*}
0&<\frac{1}{f(1)}+\sum\limits_{i=3}^n\frac{1}{f(i)}\\
&<\frac{1}{1^2+1}+\frac{1}{3^2+1}+\sum_{i=4}^n\frac{1}{i^2}\\
&=\frac{1}{1^2+1}+\frac{1}{3^2+1}-1-\frac{1}{4}-\frac{1}{9}+\sum\limits_{i=1}^n\frac{1}{i^2}\\
&<\zeta(2)-\frac{137}{180}<1.
\end{align*}
But $f(2)^2=25>20=f(1)f(3)$. Thus by Lemma \ref{lem2}, we infer that
$H_{k,f}^*(n)<H_{k-1,f}^*(n)$, which implies that $H_{k,f}^*(n)$
is decreasing as $k$ increases. Hence $0<H_{k,f}^*(n)< H_{2,f}^*(n)<1$.
The claim is true in this case.

{\sc Case 4.} $f(x)=ax+b$ with $a,b \ge1$ and $\vec{s}=(s_1, ..., s_k)$
is a $k$-tuple of positive integers such that $s_j \ge 2$ for all
integers $j$ with $1\le j\le k$. For any positive integer $r$,
we derive that $f(r)^{s_j}\ge (r+1)^2$. So we have
\begin{align}\label{eq8}
H_{k,f}^*(\vec{s},n)=\sum\limits_{1\leq i_{1}\leq \cdots\leq i_{k}\leq n}
\prod\limits_{j=1}^{k}\frac{1}{f(i_{j})^{s_j}}
\le\sum\limits_{1\leq i_{1}\leq \cdots\leq i_{k}\leq n}
\prod\limits_{j=1}^{k}\frac{1}{(i_j+1)^2}
=H_{k,g}^*(n),
\end{align}
where $g(x)=x^2+2x+1$.

But we have shown in Case 2 that $0<H_{k,q}^*(n)<1$ for all
integers $k$ with $1\le k\le n$ if $\deg q(x)\ge 2$ and $q(x)$
contains at least two terms as well as $q(x)\ne x^2+1$.
Hence $0<H_{k,g}^*(n)<1$. Then from
(\ref{eq8}) one derives that $0<H_{k,f}^*(\vec{s},n)<1$.
Thus $H_{k,f}^*(\vec{s},n)$ is not an integer in this case.

This concludes the proof of Theorem \ref{thm}.
\hfill$\Box$

\section{Proof of Theorem \ref{thm2}}

In this section, we present the proof of Theorem \ref{thm2}.\\

{\it Proof of Theorem \ref{thm2}.}
It is clear that $H_{k,f}^*(\vec{s},n)=H_{k,f}(\vec{s},n)=1$
if $n=1$. In what follows we let $n\ge 2$.

First of all, we show that $H_{k,f}^*(\vec{s},n)$ is not an integer if $n\ge 2$.
By Bertrand's postulate, we know that there exists at
least one prime $p$ such that $\frac{2n-1}{2}<p\le 2n-1$.
Then $p\le 2n-1<2p$. One can write $p:=2r-1$. Thus
\begin{align*}
&H_{k,f}^*(\vec{s},n)\\
=&\sum\limits_{1\le i_1\le \cdots\le i_k\le n
\atop \forall j, 2i_j-1=p}\frac{1}{(2i_1-1)^{s_1}\cdots (2i_k-1)^{s_k}}+
\sum\limits_{1\le i_1\le \cdots\le i_k\le n
\atop \exists j, 2i_j-1\neq p}\frac{1}{(2i_1-1)^{s_1}\cdots (2i_k-1)^{s_k}}\\
=&p^{-\sum\limits_{i=1}^ns_i}+\frac{a}{bp^t},
\end{align*}
where $\gcd(a,b)=\gcd(a,p)=\gcd(b,p)=1$ and $t<\sum\limits_{i=1}^ns_i$.
It follows that
$$v_p(H_{k,f}^*(\vec{s},n))=-\sum\limits_{i=1}^ns_i<0.$$
Thus $H_{k,f}^*(\vec{s},n)$ is not an integer if $n\ge 2$ as desired.

In the remaining part of the proof, we show that $H_{k,f}(\vec{s},n)$
is not an integer. Let $k=1$. If $s_1=1$, then it is well known that
$H_{k,f}(\vec{s},n)=\sum_{i=1}^n\frac{1}{2i-1}$ is not an integer. Now let $s_1\ge 2$.
If $i\ge 2$, then $(2i-1)^{s_1}\ge(2i-1)^2 >2i^2$. It follows that
$$1< H_{k,f}(\vec{s},n)=1+ \sum\limits_{i=2}^n\frac{1}{(2i-1)^{s_1}}<1+\frac{1}{2}
\sum\limits_{i=2}^n\frac{1}{i^2}<\frac{1}{2}+\frac{\pi^2}{12}<2,$$
which implies that $H_{k,f}(\vec{s},n)$ is not an integer if $k=1$ and $s_1\ge 2$.
Hence $H_{k,f}(\vec{s},n)$ is not an integer if $k=1$.

Subsequently, let $k \ge \frac{e}{2}\log(2n+1)+e$. Then by Lemma \ref{lem4},
$0< H_{k,f}(\vec{s},n)<1$. So $H_{k,f}(\vec{s},n)$ is not an integer.

Finally, let $2\le k < \frac{e}{2}\log(2n+1)+e$.
We divide the proof into two cases.

{\sc Case 1.} $n> 62801$. Then Lemma \ref{lem6} guarantees
the existence of a prime $p$ with
$\frac{n}{k+\frac{1}{2}}<p\le \frac{n}{k}$ and $p> 2k$.
By Lemma \ref{lem7}, we have
$v_p\big(H_{k,f}(\vec{s},n)\big)=-\sum_{i=1}^ks_i<0,$
which implies that $H_{k,f}(\vec{s},n)$ is not an integer.

{\sc Case 2.} $2\le n\le 62801$. Then $2\le k\le 18$ since
$2\le k < \frac{e}{2}\log(2n+1)+e$. Let $p_i$ denote the $i$th prime.
Note that $p_{6302}=62801$. For any integer $k$ with $2\le k\le 18$,
we let $p(k)$ denote the largest prime number $p_j$ satisfying
$kp_j\ge (k+\frac{1}{2})p_{j-1}$ and $2 \le j\le 6302$.
Define $n_k:=kp(k)-1$ for any integer $k$ with $2\le k\le 18$.
We claim that for any integer $k$ with $2\le k\le 18$,
if $n_k+1\le n\le 62801$, then there is a prime number
$p$ such that $\frac{n}{k+\frac{1}{2}}<p\le \frac{n}{k}$.
Actually, if $\frac{n}{k+\frac{1}{2}}<p(k)$, then
$\frac{n}{k+\frac{1}{2}}<p(k) \le \frac{n}{k}$
since $n_k+1\le n$. So the claim holds in this situation.
If $\frac{n}{k+\frac{1}{2}}\ge p(k)$, then we can deduce
that there is a prime $p_i\ge p(k)$ such that
$p_i \le \frac{n}{k+\frac{1}{2}}< p_{i+1}$. Then by the assumption that
$p(k)$ is the largest prime $p_j$ such that $kp_j\ge (k+\frac{1}{2})p_{j-1}$,
one derives that $kp_{i+1}<(k+\frac{1}{2})p_i$. So $kp_{i+1}<(k+\frac{1}{2})p_i\le n$
since $p_i \le \frac{n}{k+\frac{1}{2}}$. Letting $p:=p_{i+1}$ gives us
that $\frac{n}{k+\frac{1}{2}}<p< \frac{n}{k}$ as claimed. The claim is proved.

If $n\ge n_k+1$, then by the claim above, we know that there
exists a prime $p$ such that $\frac{n}{k+\frac{1}{2}}<p\le \frac{n}{k}$.
Using Maple 17, we compute the values of all the $p(k)$ and $n_k$ for $2\le k\le 18$
given in Table \ref{tbf1}. We have $p>\frac{n_k+1}{k+\frac{1}{2}}\ge 2k$
since $p$ is a prime such that
$p>\frac{n}{k+\frac{1}{2}}$ and $n\ge n_k+1$. Then applying Lemma \ref{lem7}
we know that $H_{k,f}(\vec{s},n)$ is not an integer.
So Theorem \ref{thm2} is true if $n\ge n_k+1$.

Now let $n\le n_k$. Since $2\le k\le 18$, using Maple 17,
we can calculate $H_{k,f}(n_k)$ for all integers $k$ with $2\le k\le 18$
which are given in the Table 2. From Table 2, one can read that
$H_{k,f}(n_k)$ is not an integer for all integers $k$ with $2\le k\le 18$.
Further, we can easily read from Table 2 that $H_{k,f}(n_k)<1$ if $k\ge 9$.
But $H_{k,f}(n)\le H_{k,f}(n_k)$ if $n\le n_k$, and note that
$H_{k,f}(\vec{s},n)\le H_{k,f}(n)$ for any $k$-tuple
$\vec{s}=(s_1, ..., s_k)$ of positive integers. Thus
$0<H_{k,f}(\vec{s},n)\le H_{k,f}(n)<1$ if $9\le k\le 18$ and $n\le n_k$.
In other words, $H_{k,f}(\vec{s},n)$ is not an integer
if $9\le k\le 18$ and $k\le n\le n_k$.

\begin{table}
\centering \caption{Evaluations of $p(k)$ and $n_k$ with respect
to $ 2 \le k\le 18$.}\label{tbf1}
\begin{tabular}{|c|r|r|r|r|r|r|r|r|r|}
 \hline
$k$  &2  &3 &4 &5 & 6 & 7 & 8 & 9 & 10 \\
\hline
$p(k)$  &29  & 37 & 53 & 127 & 127 & 149 & 149 & 223 & 223 \\
\hline
$n_k:=kp(k)-1$ &57 & 110 & 211 & 634 & 761 & 1042 & 1191 & 2006 & 2229 \\
\hline
$k$ & 11 & 12 & 13 & 14 & 15 & 16 & 17 & 18 & \\
\hline
$p(k)$ & 307 & 331 & 331 & 331 & 541 & 541 & 541 & 541 & \\
\hline
$n_k:=kp(k)-1$ & 3376 & 3971 & 4302 & 4633 & 8114 & 8655 & 9196 & 9737 & \\
\hline
\end{tabular}
\end{table}

\begin{table}\label{tbf2}
\centering \caption{Evaluations of $H_{k,f}(n_k)$ with respect to $ 2 \le k\le 18$.}
\begin{tabular}{|c|r|r|r|r|r|r|}
  \hline
  $k$ & 2 & 3 & 4 & 5 & 6 & 7\\
  \hline
  $H_{k,f}(n_k)$ & 3.89... & 4.46... & 4.55... & 6.16... & 3.99... & 2.61...\\
  \hline
  $k$  & 8 & 9 & 10 & 11 & 12 & 13\\
  \hline
   $H_{k,f}(n_k)$  & 1.30... & 0.95... & 0.40...&0.24... & 0.10... & 0.03...\\
  \hline
  $k$  & 14 & 15 & 16 & 17 & 18 & \\
  \hline
  $H_{k,f}(n_k)$  & 0.01... & 0.008... & 0.0024... & 0.00067... & 0.00018...& \\
  \hline

\end{tabular}
\end{table}

It remains to show that $H_{k,f}(\vec{s},n)$ is not an integer if $2\le k\le 8$
and $2\le n\le n_k$. In what follows, we let $2\le k\le 8$ and $2\le n\le n_k$.

Let $s_1=\cdots=s_k=1$. Then by Maple 17 (see the Appendix) we can compute
and find that $H_{k,f}(n)$ is not an integer.

Let $s_{j_0}\ge 2$ for some integer $j_0$ with $2\le j_0\le k$.
Define a $k$-tuple $\vec{s'}=(s_1', ..., s_k')$ of positive integers
by $s_2'=2$ and $s_1'=s_3'=\cdots=s_k'=1$. Let $i_1, ..., i_k$ be integers
such that $1\le i_1<...<i_k\le n$. If $j_0=2$, then $s_{j_0}=s_2\ge 2$ and so
$f(i_{j_0})^{s_{j_0}}=(2i_{j_0}-1)^{s_{j_0}}\ge (2i_{j_0}-1)^{2}
=f(i_{j_0})^{s'_{j_0}}>0$.
If $j_0\ge 3$, then $s_2+s_{j_0}-1\ge 2$ since $s_2\ge 1$ and $s_{j_0}\ge 2$.
It follows from $f(i_{j_0})\ge f(i_2)>0$ that
$f(i_2)^{s_2}f(i_{j_0})^{s_{j_0}}\ge f(i_2)^{s_2+s_{j_0}-1}f(i_{j_0})
\ge f(i_2)^{s'_2} f(i_{j_0})^{s'_{j_0}}>0$. Since $s_j\ge 1=s'_j$
if $j\ne 2$, we deduce that
$\prod_{j=1}^k f(i_j)^{s_j}\ge \prod_{j=1}^k f(i_j)^{s'_j}>0$.
Therefore
$$H_{k,f}(\vec{s},n)
=\sum\limits_{1\leq i_{1}<\cdots<i_{k}\le n}
\prod\limits_{j=1}^{k}\frac{1}{f(i_{j})^{s_j}}
\le\sum\limits_{1\leq i_{1}<\cdots<i_{k}\le n}
\prod\limits_{j=1}^{k}\frac{1}{f(i_{j})^{s'_j}}
=H_{k,f}(\vec{s'},n).$$
Note that $H_{k,f}(\vec{s'},n)\le H_{k,f}(\vec{s'},n_k)$ since
$n\le n_k$, and we can calculate that $H_{k,f}(\vec{s'},n_k)<1$.
Hence $0<H_{k,f}(\vec{s},n)<1$ if $s_{j_0}\ge 2$
for some integer $j_0$ with $2\le j_0\le k$.

In the following, we let $s_1\ge 2$ and $s_2=\cdots=s_k=1$.
Since $2\le k\le 8$ and $2\le n\le n_k$, one has $n\le n_8=1191$.
We observe the prime distribution less than $2381$.
Note that $p_{353}=2381$. By a simple calculation, one concludes that
$\max_{2\le i\le 353}\{p_i-p_{i-1}\}=34$.

We assert that if $48 \le n \le 1191$, then there exists
an odd prime $q_n$ depending on $n$ such that $14\le (2n-1)-q_n \le 46$. In fact,
let $q_n'$ be the largest prime less than $2n-1$. If $(2n-1)-q_n' \ge 14$,
then $(2n-1)-q_n' \le 34$. Thus $q_n=q_n'$ gives us the claim.
If $(2n-1)-q_n' < 14$, then we can find the largest prime $q_n$ such that $q_n<q_n'$
and $(2n-1)-q_n \ge 14$.
If there is no prime $q_n''$ such that $q_n<q_n''<q_n'$, then
$$14\le (2n-1)-q_n \le (2n-1)-q_n'+q_n'-q_n\le 46.$$
If there exists a prime which is less than $q_n'$ and greater that $q_n$,
then we let $q_n'''$ be the largest one. Thus $2n-1-q_n'''\le 12$.
By $\max_{2\le i\le 353}\{p_i-p_{i-1}\}=34$,
we have $q_n'''-q_n \le 34$. It then follows that
$$14\le (2n-1)-q_n \le (2n-1)-q_n'''+q_n'''-q_n\le 46.$$
as desired. This finishes the proof of the assertion.

For $2\le k\le 8$, we let $48 \le n \le n_k$. If $2\le s_1\le 29$,
then using Maple 17 (see the Appendix), we can calculate and find that
$H_{k,f}(\vec{s},n)$ is not an integer.

In the following, we let $s_1\ge 30$. Since $48\le n \le n_k \le 1191$,
we can find a prime $q_n$ such that $2n-47 \le q_n \le 2n-15$ by the
assertion. This implies that $q_n\ge 53$ and $q_n<q_n+14\le 2n-1\le q_n+46<2q_n$.
In the following, we can use such prime $q_n$ to split $H_{k,f}(\vec{s},n)$ into two parts:
$H_{k,f}(\vec{s},n)=H_1+H_2$, where
\begin{align*}
H_1=\sum\limits_{1\le i_1<\cdots<i_k\le n
\atop 2i_1-1=q_n}\prod\limits_{j=1}^k\frac{1}{(2i_j-1)^{s_j}}
\end{align*}
and
$$H_2=\sum\limits_{1\le i_1<\cdots<i_k\le n
\atop 2i_1-1 \neq q_n}\prod\limits_{j=1}^k\frac{1}{(2i_j-1)^{s_j}}.$$

We compute the $q_n$-adic valuation  of $H_2$.
Since $q_n$ is a prime such that $q_n< 2n-1<2q_n$, it is easy to see that
$v_{q_n}(H_2)\ge -1$. If one can show that $v_{q_n}(H_1)\le -2$,
then it follows that $v_{q_n}(H_{k,f}(\vec{s},n))=\min\{v_{q_n}(H_1),v_{q_n}(H_2)\}\le -2$,
which implies that $H_{k,f}(\vec{s},n)$ is not an integer.

In what follows we show that $v_{q_n}(H_1)\le -2$. By the assumption, one has
\begin{align*}
H_1=\sum\limits_{1\le i_1<\cdots<i_k\le n\atop 2i_1-1=q_n}
\prod\limits_{j=1}^k\frac{1}{(2i_j-1)^{s_j}}
=\frac{1}{q_n^{s_1}}\sum\limits_{\frac{q_n+1}{2}<i_2<\cdots<i_k\le n}
\prod\limits_{j=2}^k\frac{1}{2i_j-1}.
\end{align*}
Since $q_n$ is an odd integer and $\frac{q_n+1}{2}<i_2<\cdots<i_k\le n$,
it follows that all of $2i_2-1,\cdots,2i_k-1$ are contained in the
set $\{q_n+2,\cdots,2n-1\}$. Then one can write
\begin{align*}
H_1=\frac{H_1'}{q_n^{s_1}(q_n+2)(q_n+4)\cdots (2n-1)},
\end{align*}
where
$$H_1':=\sum\limits_{1\le i_1<\cdots<i_{n+1-k-\frac{q_n+1}{2}}
\le n-\frac{q_n+1}{2}}(q_n+2i_1)\cdots \big(q_n+2i_{n+1-k-\frac{q_n+1}{2}}\big).$$
So we conclude that there are $\dbinom{n-\frac{q_n+1}{2}}{k-1}$ terms in $H_1'$.
The above assertion tells us that $q_n+14\le 2n-1\le q_n+46$. But $2\le k\le 8$.
Then we derive that
$$\dbinom{n-\frac{q_n+1}{2}}{k-1}\le \dbinom{23}{k-1}
\le\dbinom{23}{7}=245157$$ and $$n-\frac{q_n+1}{2}-(k-1)\le 22.$$
Therefore
\begin{align}\label{eq22}
H_1'< \dbinom{n-\frac{q_n+1}{2}}{k-1}(2n-1)^{n-\frac{q_n+1}{2}-(k-1)}\le 245157(q_n+46)^{22}.
\end{align}

Let $H_1'=aq_n^t$ with $a$ coprime to $q_n$.
Note that $q_n \ge 53$. We can check that $245157(53+46)^{22}<53^{s_1-1}$.
It is easy to see that for any integer $r\ge 23$ and all real numbers $a_1$
and $a_2$ with $0<a_2<a_1$, one has
$$\Big(\frac{a_2}{a_1}\Big)^{r}<\Big(\frac{a_2}{a_1}\Big)^{22}<\Big(\frac{a_2+46}{a_1+46}\Big)^{22},$$
which implies that $\frac{(a_1+46)^{22}}{a_1^{r}}<\frac{(a_2+46)^{22}}{a_2^{r}}$.
In particular, if $s_1\ge 30$, then
$$\frac{245157(q_n+46)^{22}}{q_n^{s_1-1}}\le \frac{245157(53+46)^{22}}{53^{s_1-1}}<1.$$
So $245157(q_n+46)^{22}<q_n^{s_1-1}$. Thus by (\ref{eq22}), we have
$q_n^t\le H_1'=aq_n^t<q_n^{s_1-1}$ which implies that $ v_{q_n}(H_1')=t\le s_1-2$.
Hence $v_{q_n}(H_1)\le -2$ if $s_1\ge 30$. Therefore $H_{k,f}(\vec{s},n)$
is not an integer if $48\le n \le n_k$ and $s_1\ge 30$.

Finally, we let $k\le n\le 47$.
If $5\le k\le 8$, then we can read from Table 3 that $0< H_{k,f}(47)<1$.
So $0<H_{k,f}(n)<1$ for $k\le n \le 47$. It follows that $0< H_{k,f}(\vec{s},n)<1$.
Namely, $H_{k,f}(\vec{s},n)$ is not an integer if $5\le k\le 8$ and $k\le n\le 47$.
\begin{table}
\centering \caption{Evaluations of $H_{k,f}(47)$ with respect
to $ 5 \le k\le 8$.}\label{tbf3}
\begin{tabular}{|r|r|r|r|r|}
  \hline
  $k$ & 5 & 6 & 7 & 8 \\
  \hline
  $H_{k,f}(47)$ & 0.49... & 0.14... & 0.032... & 0.0060... \\
  \hline
\end{tabular}
\end{table}

Now let $2\le k \le 4$ and $12\le n\le 47$.
If $2\le s_1\le 6$, then using Maple 17, we can compute
and find that $H_{k,f}(\vec{s},n)$ is not an integer.
In the following, we suppose that $s_1\ge 7$.
Since $12\le n\le 47$, there always
exists an odd prime $l_n$ such that $2n-11\le l_n\le 2n-7$.
It then follows that $13\le l_n\le 83$ and
$$l_n<l_n+6\le 2n-1\le l_n+10<2l_n.$$
Using such prime $l_n$, we divide the sum $H_{k,f}(\vec{s},n)$ into two parts:
$H_{k,f}(\vec{s},n)=H_3+H_4$, where
\begin{align*}
H_3:=\sum\limits_{1\le i_1<\cdots<i_k\le n
\atop 2i_1-1=l_n}\prod\limits_{j=1}^k\frac{1}{(2i_j-1)^{s_j}}
\end{align*}
and
$$H_4:=\sum\limits_{1\le i_1<\cdots<i_k\le n
\atop 2i_1-1 \neq l_n}\prod\limits_{j=1}^k\frac{1}{(2i_j-1)^{s_j}}.$$
Clearly, we can find that $v_{l_n}(H_4)\ge -1$ since $l_n< 2n-1<2l_n$.
If $v_{l_n}(H_3)\le -2$, then
$v_{l_n}(H_{k,f}(\vec{s},n))=\min\{v_{l_n}(H_3),v_{l_n}(H_4)\}\le -2<0$.
This implies that $H_{k,f}(\vec{s},n)$ is not an integer.
In the following, we compute the $l_n$-adic valuation of $H_3$.
Evidently, $H_3$ can be written as:
$$H_3=\frac{H'_3}{l_n^{s_1}(l_n+2)\cdots (2n-1)},$$
where $$H'_3:=\sum\limits_{1\le i_1<\cdots<i_{n+1-k-
\frac{l_n+1}{2}}\le n-\frac{l_n+1}{2}}(l_n+2i_1)\cdots(l_n+2i_{n+1-k-\frac{l_n+1}{2}}).$$
Since $l_n+6\le 2n-1\le l_n+10$ and $2\le k\le 4$, one has
$3 \le n-\frac{l_n+1}{2}\le 5$ and $n+1-k-\frac{l_n+1}{2}\le 4$.
Then we obtain that
\begin{align}\label{eq23}
H'_3<\dbinom{5}{2}(l_n+10)^{n+1-k-\frac{l_n+1}{2}}\le 10(l_n+10)^4.
\end{align}
Let $H'_3:=bl_n^m$. It is easy to check that if $13\le l_n\le 83 $,
then $10(l_n+10)^4<l_n^{s_1-1}$ for $s_1\ge 7$. By (\ref{eq23}),
we have $l_n^m\le H'_3< l_n^{s_1-1}$, which implies that
$m=v_{l_n}(H'_3)\le s_1-2$. Hence $v_{l_n}(H_3)\le -2$ if $s_1\ge 7$.
So we conclude that $v_{l_n}(H_{k,f}(\vec{s},n))\le -2$.
In other words, $H_{k,f}(\vec{s},n)$ is not an integer if $s_1\ge 7$.

Last, we need to show that if $2\le k\le 4$ and $n\le 11$,
then $H_{k,f}(\vec{s},n)$ is not an integer.
If $k=3,4$, a direct computation yields that $H_{k,f}(n)<1$,
which tells us that $0<H_{k,f}(\vec{s},n)\le H_{k,f}(n)<1$ and so
$H_{k,f}(\vec{s},n)$ is not an integer if $k=3$ and 4.
Now let $k=2$. Obviously, $H_{k,f}(\vec{s},n)$ is not an integer
if $n=2$. Now let $n\in\{3, 4, 5, 7, 8, 10, 11\}$.
Then $Q_n:=2n-3$ is a prime number and $Q_n \ge 3$.
Let us now split $H_{k,f}(\vec{s},n)$ into the following two sums:
$$H_{k,f}(\vec{s},n)=H_5+H_6,$$
where
$$H_5:=\frac{1}{Q_n^{s_1}(Q_n+2)} \ \ {\rm and} \ \
H_6:=\sum\limits_{1\le i_1<i_2\le n \atop i_1
\neq\frac{Q_n+1}{2}}\frac{1}{(2i_1-1)^{s_1}(2i_2-1)}.$$
Evidently, $Q_n\nmid (Q_n+2)$ since $Q_n\ge 3$.
Since $s_1\ge 2$, we have $v_{Q_n}(H_5)=-s_1\le -2$
and $v_{Q_n}(H_6)\ge -1$. It then follows that
$v_{Q_n}(H_{k,f}(\vec{s},n))=-s_1\le -2$, which implies
that $H_{k,f}(\vec{s},n)$ is not an integer if
$n\in \{3,4,5,7,8,10, 11\}$ and $k=2$.
If $n=6$, then it is easy to check that
$H_{k,f}(\vec{s},6)$ is not an integer if $s_1=2$.
If $s_1\ge 3$, then
$$H_{k,f}(\vec{s},6):=\frac{1}{9^{s_1}\cdot11}+H_7,$$
where
$$H_7:=\sum\limits_{1\le i_1<i_2\le 6
\atop i_1\neq5}\frac{1}{(2i_1-1)^{s_1}(2i_2-1)}.$$
Clearly, one has
$$H_7=\frac{1}{3^{s_1+2}}+\frac{1}{3^{s_1}}\sum\limits_{3\le i_2\le 6\atop i_2\neq5}\frac{1}{2i_2-1}
+\frac{1}{3}+
\frac{1}{9}\sum\limits_{1\le i_1\le 4\atop i_1\neq 2}\frac{1}{(2i_1-1)^{s_1}}+
\sum\limits_{1\le i_1<i_2\le 6\atop i_1,i_2\notin \{2,5\}}\frac{1}{(2i_1-1)^{s_1}(2i_2-1)}.$$
Then we can find that $v_3(H_7)= -(s_1+2)$. It is clear that $v_3(\frac{1}{9^{s_1}\cdot11})=-2s_1$.
By $s_1\ge 3$, one derives that $-(s_1+2)>-2s_1$. It then follows that
$$v_3(H_{k,f}(\vec{s},6))= \min\{v_3\Big(\frac{1}{9^{s_1}\cdot11}\Big), v_3(H_7)\}=-2s_1<0.$$
Then $H_{k,f}(\vec{s},6)$ is not an integer.

Let $n=9$. First, we can easily check that $H_{k,f}(\vec{s},9)$ is not an integer if $s_1=2$.
Next, we let $s_1\ge 3$. Then $2i_1-1$ and $2i_2-1$ are contained in the
set $\{1,3,5,7,9,11,13,15,17\}$. Hence we can write $H_{k,f}(\vec{s},9)$ as follows:
$$H_{k,f}(\vec{s},9)=\frac{1}{9^{s_1}\cdot15}+H_8,$$
where
\begin{align*}
H_8=&\frac{1}{3}+\frac{1}{3^{s_1}}\Big(\frac{1}{9}+\frac{1}{15}\Big)+
\frac{1}{3^{s_1}}\sum\limits_{3\le i_2\le 9\atop i_2\neq 5,8}\frac{1}{2i_2-1}+
\frac{1}{9^{s_1}}\sum\limits_{6\le i_2\le 9\atop i_2\neq 8}\frac{1}{2i_2-1}
+\frac{1}{9}\sum\limits_{1\le i_1\le 4\atop i_1\neq 2}\frac{1}{(2i_1-1)^{s_1}}\\
&+\frac{1}{15^{s_1}\cdot 17}+\frac{1}{15}\sum\limits_{1\le i_1\le 7\atop i_1\neq 2,5}\frac{1}{(2i_1-1)^{s_1}}
+\sum\limits_{1\le i_1<i_2\le 9\atop i_1,i_2\notin \{2, 5, 8\}}\frac{1}{(2i_1-1)^{s_1}(2i_2-1)}.
\end{align*}
Since $s_1+2< 2s_1$, it then follows that
$$v_3(H_8)=\min \{v_3\Big(\frac{1}{3^{s_1+2}}\Big), v_3\Big(\frac{1}{9^{s_1}}
\sum\limits_{6\le i_2\le 9\atop i_2\neq 8}\frac{1}{2i_2-1}\Big)\}=-2s_1.$$
So one derives that $v_3(H_8)\ge -2s_1$ if $s_1\ge 3$.
It is not hard to see that $v_3\big(\frac{1}{9^{s_1}\cdot15}\big)=-2s_1-1$.
Hence one concludes that if $s_1\ge 3$, then
$$v_3(H_{k,f}(\vec{s},9))=\min\{v_3\Big(\frac{1}{9^{s_1}\cdot15}\Big), v_3(H_8)\}=-2s_1-1<0,$$
from which it follows that $H_{k,f}(\vec{s},9)$ is not an integer.
This finishes the proof of Theorem \ref{thm2}.

\section{Final remark}
Let $n$ and $k$ be integers with $1\le k\le n$ and
$f(x)$ be a nonzero polynomial of nonnegative integer coefficients.
Let $\vec{s}=(s_1, ..., s_k)$ be a $k$-tuple of positive integers.
Then from Theorems 1.1 and 1.2 of this paper
and the results presented in \cite{[HW]} and \cite{[KH]},
one can read that both of $H_{k,n}(\vec{s},f)$
and $H_{k,n}^*(\vec{s},f)$ are almost non-integers if deg$f(x)\ge 2$,
or deg$f(x)=1$ and $s_i\ge 2$ for all integers $i$ with $1\le i\le k$,
or deg$f(x)=1$ and $s_i=1$ for all integers $i$ with $1\le i\le k$,
or $f(x)\in \{x, 2x-1\}$. But if deg$f(x)=1$, $f(x)\not\in \{x, 2x-1\}$
and there are indexes $i$ and $j$ between 1 and $k$ such that $s_i=1$
and $s_j\ge 2$, then does the similar result hold for both of
$H_{k,f}(\vec{s}, n)$ and $H_{k,f}^*(\vec{s}, n)$? Unfortunately,
this problem seems hard to answer in general and is still kept open so far.

In the following, we let $f(x)$ be a nonzero polynomial of integer
coefficients. Let $\mathbb{Z}$ and $\mathbb{Z}^+$ be the set of
integers and the set of positive integers, respectively. Let
$Z_f:=\{x\in \mathbb{Z}: f(x)=0\}$ be the set of integer roots of $f(x)$
and $\{a_k\}_{k=1}^{\infty}:=\mathbb{Z}^+\setminus Z_f$ be arranged in the
increasing order. Then $f(a_k)\ne 0$ for all integers $k\ge 1$. Define
$$M_{k,f}(\vec{s}, n):=\sum\limits_{1\leq i_{1}<\cdots<i_{k}\le n}
\prod\limits_{j=1}^{k}\frac{1}{f(a_{i_{j}})^{s_j}}$$
and
$$M_{k,f}^*(\vec{s}, n):=\sum\limits_{1\leq i_{1}\le\cdots\le i_{k}\le n}
\prod\limits_{j=1}^{k}\frac{1}{f(a_{i_{j}})^{s_j}}.$$
If $Z_f$ is empty, then $M_{k,f}(\vec{s}, n)$ and $M_{k,f}^*(\vec{s}, n)$
become $H_{k,f}(\vec{s}, n)$ and $H_{k,f}^*(\vec{s}, n)$, respectively.

On the one hand, for any given integer $N_0\ge 1$, one can easily
find a polynomial $f_0(x)$ of integer coefficients such that for all
integers $n$ and $k$ with $1\le k\le n\le N_0$ and for any $k$-tuple
$\vec{s}=(s_1, ..., s_k)$ of positive integers, both of
$H_{k,f_0}(\vec{s}, n)$ and $H_{k,f_0}^*(\vec{s}, n)$
are integers. Actually, letting
$$f_0(x)=\prod_{i=1}^{N_0}(x-i)\pm 1$$
gives us the desired result.
On the other hand, for any given nonzero polynomial $f(x)$ of
integer coefficients, we believe that the similar integrality result
is still true. So in concluding this paper, we propose the following
more general conjecture that generalizes Conjecture 3.1 of \cite{[LHQW]}.

\begin{con}
Let $f(x)$ be a nonzero polynomial of integer coefficients and
$\{s_i\}_{i=1}^\infty$ be an infinite sequence of positive integers
(not necessarily increasing and not necessarily distinct).
Then there is a positive integer $N$ such that for any integer
$n\ge N$ and for all integers $k$ with $1\le k\le n$, both of
$M_{k,f}(\vec{s}^{(k)}, n)$ and $M_{k,f}^*(\vec{s}^{(k)}, n)$
are not integers, where $\vec{s}^{(k)}:=(s_1, ..., s_k)$ is the
$k$-tuple formed by the first $k$ terms of the sequence
$\{s_i\}_{i=1}^\infty$.
\end{con}

Obviously, the results presented in \cite{[CT]}, \cite{[EN]}-\cite{[HW]},
\cite{[LHQW]}-\cite{[WH]} and Theorems 1.1 and 1.2
of this paper provide evidences to Conjecture 5.1.\\

\begin{center}
Acknowledgement
\end{center}
The authors would like to thank the anonymous referee for the careful reading of
the manuscript and helpful comments.\\

\begin{center}
{\sc Appendix}
\end{center}

with(linalg): with(numtheory): with(combinat, choose):$R$:=matrix(1,20):\\
$n:=[0,57, 110, 211, 634, 761, 1042, 1191]$:\\
for $k$ from 2 to 8 do\\
 print($k$);\\
for $m$ from 1 to 29 do\\
$s:=[m,1,1{\$}(k-2)]$:\\
$S$:=vector($k$,0);\\
for $i$ from $k$ to $n[k]$ do\\
$S[1]:=S[1]+1/(2*i-2*k+1)^{s[1]}$;\\
for $j$ from 2 to $k$ do\\
$S[j]:=S[j]+S[j-1]/(2*i-2*k+2*j-1)^{s[j]}$;\\
if type(($S[j]$,integer)) then\\
print($i-k+j$,$j$*IsInt)\\
end  if: end do:\\
end do:\\
od:od:

\bibliographystyle{amsplain}

\end{document}